\documentclass{amsart}
\usepackage{amsfonts}

\newtheorem{theorem}{Theorem}
\theoremstyle{plain}

\newtheorem{lemma}{Lemma}

\newtheorem{proposition}{Proposition}
\newtheorem{remark}{Remark}

\numberwithin{equation}{section}
\input{tcilatex}

\begin{document}
\title[Integral Inequalities]{SOME GENERALIZATIONS OF INTEGRAL INEQUALITIES
AND THEIR APPLICATIONS}
\author{Mustafa G\"{U}RB\"{U}Z}
\address{Agri Ibrahim Cecen University\\
Education Faculty Agri/TURKEY}
\email{mgurbuz@agri.edu.tr}
\author{Abdullah YARADILMI\c{S}}
\address{Agri Ibrahim Cecen University\\
Education Faculty Agri/TURKEY}
\email{ayaradlms@gmail.com}
\date{March 08, 2015}

\begin{abstract}
In this paper, an integral identity for twice differentiable functions is
generalized. Then, by using convexity of $\left\vert f^{\prime \prime
}\right\vert $ or $\left\vert f^{\prime \prime }\right\vert ^{q}$ and with
the aid of power mean and H\"{o}lder's inequalities we achieved some new
results. We also gave some applications to quadrature formulas and some
special means. Therewithal, by choosing $\alpha =\frac{1}{2}$ in our main
results, we obtained some findings in [13].
\end{abstract}

\maketitle

\section{INTRODUCTION}

\bigskip We shall recall the definitions of convex functions:

Let $I$ be an interval in $%
\mathbb{R}
$. Then $f:I\rightarrow 
\mathbb{R}
$ is said to be convex if for all $x,y\in I$ and all $\alpha \in \left[ 0,1%
\right] ,$%
\begin{equation}
f\left( \alpha x+\left( 1-\alpha \right) y\right) \leq \alpha f\left(
x\right) +\left( 1-\alpha \right) f\left( y\right)  \label{con}
\end{equation}%
holds. If (\ref{con}) is strict for all $x\neq y$ and $\alpha \in \left(
0,1\right) $, then $f$ is said to be strictly convex. If the inequality in (%
\ref{con}) is reversed, then $f$ is said to be concave. If it is strict for
all $x\neq y$ and $\alpha \in \left( 0,1\right) $, then $f$ is said to be
strictly concave.

The following inequality is called Hermite-Hadamard inequality for convex
functions:

Let $f:I\subseteq \mathbb{R\rightarrow R}$ be a convex mapping defined on
the interval $I$ of real numbers and $a,b\in I$ with $a<b.$ The following
double inequality:%
\begin{equation}
f(\frac{a+b}{2})\leq \frac{1}{b-a}\int\limits_{a}^{b}f(x)dx\leq \frac{%
f(a)+f(b)}{2}  \label{0}
\end{equation}%
is known in the literature as Hadamard inequality for convex mapping. Note
that some of the classical inequalities for means can be derived from (\ref%
{0}) for appropriate particular selections of the mapping $f.$ Both
inequalities hold in the reversed direction if $f$ is concave.

It is well known that the Hermite-Hadamard's inequality plays an important
role in nonlinear analysis. Over the last decade, this classical inequality
has been improved and generalized in a number of ways; there have been a
large number of research papers written on this subject, (see, \cite{SSDRPA}-%
\cite{yeni}) and the references therein.

In \cite{ana}, Sarikaya \textit{et. al.} established inequalities for twice
differentiable convex mappings which are connected with Hadamard's
inequality. They used the following lemma and proved next two theorems:

\begin{lemma}
\label{s1}Let $I\subset \mathbb{R}$ be an open interval, $a,b\in I$ with $%
a<b.$ If $f:I\rightarrow \mathbb{R}$ is a twice differentiable mapping such
that \ $f^{\prime \prime }$is integrable and $0\leq \lambda \leq 1.$ Then
the following identity holds:%
\begin{equation}
(\lambda -1)f(\frac{a+b}{2})-\lambda \frac{f(a)+f(b)}{2}+\dfrac{1}{b-a}%
\dint_{a}^{b}f(x)dx=\left( b-a\right) ^{2}\dint_{0}^{1}k(t)f^{\prime \prime
}(ta+(1-t)b)dt  \label{g1}
\end{equation}%
where 
\begin{equation*}
k(t)=\left\{ 
\begin{array}{ll}
t(t-\lambda )/2, & 0\leq t\leq 1/2 \\ 
(1-t)(1-\lambda -t)/2, & 1/2\leq t\leq 1.%
\end{array}%
\right.
\end{equation*}
\end{lemma}

\begin{theorem}
\label{s2}Let $I\subset \mathbb{R}$ be an open interval, $a,b\in I$ with $%
a<b $ and $f:I\rightarrow \mathbb{R}$ be a twice differentiable mapping such
that \ $f^{\prime \prime }$is integrable and $0\leq \lambda \leq 1.$ If $%
\left\vert f^{\prime \prime }\right\vert $ is a convex on $\left[ a,b\right]
,$ then the following inequalities hold:%
\begin{equation}
\begin{array}{l}
\left\vert (\lambda -1)f(\dfrac{a+b}{2})-\lambda \dfrac{f(a)+f(b)}{2}+\dfrac{%
1}{b-a}\dint_{a}^{b}f(x)dx\right\vert \\ 
\leq \left\{ 
\begin{array}{ll}
\dfrac{\left( b-a\right) ^{2}}{12}\left[ \left( \lambda ^{4}+\left(
1+\lambda \right) (1-\lambda )^{3}+\dfrac{5\lambda -3}{4}\right) \left\vert
f^{\prime \prime }\left( a\right) \right\vert \right. &  \\ 
\left. +\left( \lambda ^{4}+\left( 2-\lambda \right) \lambda ^{3}+\dfrac{%
1-3\lambda }{4}\right) \left\vert f^{\prime \prime }\left( b\right)
\right\vert \right] , & \text{for }0\leq \lambda \leq \frac{1}{2} \\ 
\dfrac{\left( b-a\right) ^{2}\left( 3\lambda -1\right) }{48}\left[
\left\vert f^{\prime \prime }\left( a\right) \right\vert +\left\vert
f^{\prime \prime }\left( b\right) \right\vert \right] & \text{for }\frac{1}{2%
}\leq \lambda \leq 1.%
\end{array}%
\right.%
\end{array}
\label{g2}
\end{equation}
\end{theorem}

\begin{theorem}
\label{s3}Let $I\subset \mathbb{R}$ be an open interval, $a,b\in I$ with $%
a<b $ and $f:I\rightarrow \mathbb{R}$ be a twice differentiable mapping such
that \ $f^{\prime \prime }$is integrable and $0\leq \lambda \leq 1.$ If $%
\left\vert f^{\prime \prime }\right\vert ^{q}$ is a convex on $\left[ a,b%
\right] ,\ q\geq 1,$ then the following inequalities hold:%
\begin{equation}
\begin{array}{l}
\left\vert (\lambda -1)f(\dfrac{a+b}{2})-\lambda \dfrac{f(a)+f(b)}{2}+\dfrac{%
1}{b-a}\dint_{a}^{b}f(x)dx\right\vert \\ 
\\ 
\leq \left\{ 
\begin{array}{ll}
\dfrac{\left( b-a\right) ^{2}}{2}\left( \dfrac{\lambda ^{3}}{3}+\dfrac{%
1-3\lambda }{24}\right) ^{1-\frac{1}{q}} & \  \\ 
&  \\ 
\times \left\{ \left( \left[ \dfrac{\lambda ^{4}}{6}+\dfrac{3-8\lambda }{%
3\times 2^{6}}\right] \left\vert f^{\prime \prime }(a)\right\vert ^{q}+\left[
\dfrac{\left( 2-\lambda \right) \lambda ^{3}}{6}+\dfrac{5-16\lambda }{%
3\times 2^{6}}\right] \left\vert f^{\prime \prime }(b)\right\vert
^{q}\right) ^{\frac{1}{q}}\right. &  \\ 
&  \\ 
+\left. \left( \left[ \dfrac{1+\lambda }{6}(1-\lambda )^{3}+\dfrac{48\lambda
-27}{3\times 2^{6}}\right] \left\vert f^{\prime \prime }(a)\right\vert ^{q}+%
\left[ \dfrac{\lambda ^{4}}{6}+\dfrac{3-8\lambda }{3\times 2^{6}}\right]
\left\vert f^{\prime \prime }(b)\right\vert ^{q}\right) ^{\frac{1}{q}%
}\right\} , & \text{for }0\leq \lambda \leq \frac{1}{2} \\ 
&  \\ 
\dfrac{\left( b-a\right) ^{2}}{2}\left( \dfrac{3\lambda -1}{24}\right) ^{1-%
\frac{1}{q}}\left\{ \left( \dfrac{8\lambda -3}{3\times 2^{6}}\left\vert
f^{\prime \prime }(a)\right\vert ^{q}+\dfrac{16\lambda -5}{3\times 2^{6}}%
\left\vert f^{\prime \prime }(b)\right\vert ^{q}\right) ^{\frac{1}{q}}\right.
&  \\ 
&  \\ 
\left. +\left( \dfrac{16\lambda -5}{3\times 2^{6}}\left\vert f^{\prime
\prime }(a)\right\vert ^{q}+\dfrac{8\lambda -3}{3\times 2^{6}}\left\vert
f^{\prime \prime }(b)\right\vert ^{q}\right) ^{\frac{1}{q}}\right\} & \text{%
for }\frac{1}{2}\leq \lambda \leq 1,%
\end{array}%
\right.%
\end{array}
\label{g3}
\end{equation}%
\ where $\frac{1}{p}+\frac{1}{q}=1$.
\end{theorem}

For some recent results connected with twice differentiable functions, see (%
\cite{ana}-\cite{tunc}).

In this paper, we achieved an integral identity for twice differentiable
functions. Then, by using convexity of $\left\vert f^{\prime \prime
}\right\vert $ or $\left\vert f^{\prime \prime }\right\vert ^{q}$ we
achieved some new results. We also gave some applications to quadrature
formulas and some special means.

\section{MAIN RESULTS}

In order to prove our theorems we need following Lemma:

\begin{lemma}
\label{aaa}Let $f:I\subset 
\mathbb{R}
\rightarrow 
\mathbb{R}
$ be a twice differentiable mapping on $I^{\circ }$ such that $f,f^{\prime
},f^{\prime \prime }\in L\left[ a,b\right] $, where $a,b\in I$ with $a<b$
and $\alpha ,\lambda \in \left[ 0,1\right] $. Then the following equality
holds:%
\begin{eqnarray}
&&\left( b-a\right) \left( \alpha -\frac{1}{2}\right) f^{\prime }\left(
\left( 1-\alpha \right) b+\alpha a\right) -\frac{1}{b-a}\int_{a}^{b}f\left(
x\right) dx  \label{r1} \\
&&+\left( 1-\lambda \right) f\left( \left( 1-\alpha \right) b+\alpha
a\right) +\lambda \left( \alpha f\left( a\right) +\left( 1-\alpha \right)
f\left( b\right) \right)  \notag \\
&=&\frac{\left( b-a\right) ^{2}}{2}\int_{0}^{1}k\left( t\right) f^{\prime
\prime }\left( tb+\left( 1-t\right) a\right) dt  \notag
\end{eqnarray}%
where%
\begin{equation*}
k\left( t\right) =\left\{ 
\begin{array}{c}
\text{ \ \ \ \ \ \ \ \ \ }2\alpha \lambda t-t^{2}\text{, \ \ \ \ \ \ \ \ \ \
\ \ \ }0\leq t\leq 1-\alpha \\ 
\left( 1-t\right) \left( t-1+2\lambda \left( 1-\alpha \right) \right) \text{%
, \ \ \ }1-\alpha \leq t\leq 1.%
\end{array}%
\right.
\end{equation*}
\end{lemma}

\begin{proof}
We note that%
\begin{eqnarray*}
I &=&\int_{0}^{1-\alpha }\left( 2\alpha \lambda t-t^{2}\right) f^{\prime
\prime }\left( tb+\left( 1-t\right) a\right) dt \\
&&+\int_{1-\alpha }^{1}\left( 1-t\right) \left( t-1+2\lambda \left( 1-\alpha
\right) \right) f^{\prime \prime }\left( tb+\left( 1-t\right) a\right) dt.
\end{eqnarray*}%
Integrating by parts, we get%
\begin{eqnarray*}
I &=&\left. \left( 2\alpha \lambda t-t^{2}\right) \frac{f^{\prime }\left(
tb+\left( 1-t\right) a\right) }{b-a}\right\vert _{0}^{1-\alpha } \\
&&-\int_{0}^{1-\alpha }\left( 2\alpha \lambda -2t\right) \frac{f^{\prime
}\left( tb+\left( 1-t\right) a\right) }{b-a}dt \\
&&+\left. \left( \frac{\left( 1-t\right) \left( t-1+2\lambda \left( 1-\alpha
\right) \right) f^{\prime }\left( tb+\left( 1-t\right) a\right) }{\left(
b-a\right) }\right) \right\vert _{1-\alpha }^{1} \\
&&-\int_{1-\alpha }^{1}2\left( 1-t-\lambda \left( 1-\alpha \right) \right) 
\frac{f^{\prime }\left( tb+\left( 1-t\right) a\right) }{b-a}dt \\
&=&\left( 2\alpha \lambda \left( 1-\alpha \right) -\left( 1-\alpha \right)
^{2}\right) \frac{f^{\prime }\left( \left( 1-\alpha \right) b+\alpha
a\right) }{b-a} \\
&&-\left( \alpha \left( -\alpha +2\lambda \left( 1-\alpha \right) \right)
\right) \frac{f^{\prime }\left( \left( 1-\alpha \right) b+\alpha a\right) }{%
b-a} \\
&&-\frac{2}{b-a}\left\{ 
\begin{array}{c}
\int_{0}^{1-\alpha }\left( \alpha \lambda -t\right) f^{\prime }\left(
tb+\left( 1-t\right) a\right) dt \\ 
+\int_{1-\alpha }^{1}\left( 1-t-\lambda \left( 1-\alpha \right) \right)
f^{\prime }\left( tb+\left( 1-t\right) a\right) dt%
\end{array}%
\right\} .
\end{eqnarray*}%
By simple calculation we have%
\begin{eqnarray*}
I &=&\left( 2\alpha -1\right) \frac{f^{\prime }\left( \left( 1-\alpha
\right) b+\alpha a\right) }{b-a} \\
&&-\frac{2}{b-a}\left\{ 
\begin{array}{c}
\left. \left( \alpha \lambda -t\right) \frac{f\left( tb+\left( 1-t\right)
a\right) }{b-a}\right\vert _{0}^{1-\alpha }-\int_{0}^{1-\alpha }\frac{%
-f\left( tb+\left( 1-t\right) a\right) }{b-a}dt \\ 
+\left. \left( 1-t-\lambda \left( 1-\alpha \right) \right) \frac{f\left(
tb+\left( 1-t\right) a\right) }{b-a}\right\vert _{1-\alpha
}^{1}-\int_{1-\alpha }^{1}\frac{-f\left( tb+\left( 1-t\right) a\right) }{b-a}%
dt%
\end{array}%
\right\} \\
&=&\left( 2\alpha -1\right) \frac{f^{\prime }\left( \left( 1-\alpha \right)
b+\alpha a\right) }{b-a} \\
&&-\frac{2}{b-a}\left\{ 
\begin{array}{c}
\left( \alpha \lambda -1+\alpha \right) \frac{f\left( \left( 1-\alpha
\right) b+\alpha a\right) }{b-a}-\frac{\alpha \lambda f\left( a\right) }{b-a}%
+\left( \frac{-\lambda \left( 1-\alpha \right) f\left( b\right) }{b-a}\right)
\\ 
-\left( \alpha -\lambda \left( 1-\alpha \right) \right) \frac{f\left( \left(
1-\alpha \right) b+\alpha a\right) }{b-a}+\int_{0}^{1}\frac{f\left(
tb+\left( 1-t\right) a\right) }{b-a}dt%
\end{array}%
\right\} \\
&=&\left( 2\alpha -1\right) \frac{f^{\prime }\left( \left( 1-\alpha \right)
b+\alpha a\right) }{b-a} \\
&&+\frac{2}{\left( b-a\right) ^{2}}\left\{ 
\begin{array}{c}
\left( 1-\lambda \right) f\left( \left( 1-\alpha \right) b+\alpha a\right)
\\ 
+\lambda \left( \alpha f\left( a\right) +\left( 1-\alpha \right) f\left(
b\right) \right) -\int_{0}^{1}f\left( tb+\left( 1-t\right) a\right) dt%
\end{array}%
\right\} .
\end{eqnarray*}%
By change of variable $x=tb+\left( 1-t\right) a$ and multiplying both sides
with $\frac{\left( b-a\right) ^{2}}{2}$ we get the desired result.
\end{proof}

\begin{theorem}
\label{ggg}Let $f:I\subset 
\mathbb{R}
\rightarrow 
\mathbb{R}
$ be a twice differentiable mapping on $I^{\circ }$ such that $f,f^{\prime
},f^{\prime \prime }\in L\left[ a,b\right] $, where $a,b\in I$ with $a<b$
and $\alpha ,\lambda \in \left[ 0,1\right] $. If $\left\vert f^{\prime
\prime }\right\vert $ is convex on $\left[ a,b\right] ,$ then the following
inequalities hold:%
\begin{gather}
\left\vert 
\begin{array}{c}
\left( b-a\right) \left( \alpha -\frac{1}{2}\right) f^{\prime }\left( \left(
1-\alpha \right) b+\alpha a\right) -\frac{1}{b-a}\int_{a}^{b}f\left(
x\right) dx \\ 
\\ 
+\left( 1-\lambda \right) f\left( \left( 1-\alpha \right) b+\alpha a\right)
+\lambda \left( \alpha f\left( a\right) +\left( 1-\alpha \right) f\left(
b\right) \right)%
\end{array}%
\right\vert  \label{b1} \\
\notag \\
\leq \left\{ 
\begin{array}{c}
\frac{\left( b-a\right) ^{2}}{2}\left\{ \left( {\small \gamma }_{{\small 1}}+%
{\small \mu }_{{\small 1}}\right) \left\vert {\small f}^{{\small \prime
\prime }}\left( {\small b}\right) \right\vert {\small +}\left( {\small %
\gamma }_{{\small 2}}+{\small \mu }_{{\small 2}}\right) \left\vert {\small f}%
^{{\small \prime \prime }}\left( {\small a}\right) \right\vert \right\} ,%
\text{ \ \ \ \ \ \ }{\small 2\alpha \lambda \leq 1-\alpha \leq 1-2\lambda }%
\left( {\small 1-\alpha }\right) \\ 
\\ 
\frac{\left( b-a\right) ^{2}}{2}\left\{ \left( {\small \gamma }_{{\small 1}}+%
{\small \mu }_{{\small 3}}\right) \left\vert {\small f}^{{\small \prime
\prime }}\left( {\small b}\right) \right\vert {\small +}\left( {\small %
\gamma }_{{\small 2}}+{\small \mu }_{{\small 4}}\right) \left\vert {\small f}%
^{{\small \prime \prime }}\left( {\small a}\right) \right\vert \right\} ,%
\text{ }{\small 1-\alpha \geq }\max \left\{ {\small 2\alpha \lambda
,1-2\lambda }\left( {\small 1-\alpha }\right) \right\} \\ 
\\ 
\frac{\left( b-a\right) ^{2}}{2}\left\{ \left( {\small \gamma }_{{\small 3}}+%
{\small \mu }_{{\small 1}}\right) \left\vert {\small f}^{{\small \prime
\prime }}\left( {\small b}\right) \right\vert {\small +}\left( {\small %
\gamma }_{{\small 4}}+{\small \mu }_{{\small 2}}\right) \left\vert {\small f}%
^{{\small \prime \prime }}\left( {\small a}\right) \right\vert \right\} ,%
\text{\ }{\small 1-\alpha \leq }\min \left\{ {\small 2\alpha \lambda
,1-2\lambda }\left( {\small 1-\alpha }\right) \right\} \\ 
\\ 
\frac{\left( b-a\right) ^{2}}{2}\left\{ \left( {\small \gamma }_{{\small 3}}+%
{\small \mu }_{{\small 3}}\right) \left\vert {\small f}^{{\small \prime
\prime }}\left( {\small b}\right) \right\vert {\small +}\left( {\small %
\gamma }_{{\small 4}}+{\small \mu }_{{\small 4}}\right) \left\vert {\small f}%
^{{\small \prime \prime }}\left( {\small a}\right) \right\vert \right\} ,%
\text{ \ \ \ \ \ }{\small 1-2\lambda }\left( {\small 1-\alpha }\right) 
{\small \leq 1-\alpha \leq 2\alpha \lambda },%
\end{array}%
\right.  \notag
\end{gather}%
where%
\begin{eqnarray*}
\gamma _{1} &=&\frac{8}{3}\left( \alpha \lambda \right) ^{4}+\left( 1-\alpha
\right) ^{3}\left( \frac{1-\alpha }{4}-\frac{2\alpha \lambda }{3}\right) \\
\gamma _{2} &=&\frac{8}{3}\left( \alpha \lambda \right) ^{3}\left( 1-\alpha
\lambda \right) +\left( 1-\alpha \right) ^{2}\left[ \frac{1-\alpha }{3}%
-\alpha \lambda -\frac{\left( 1-\alpha \right) ^{2}}{4}+\frac{2\alpha
\lambda \left( 1-\alpha \right) }{3}\right] \\
\gamma _{3} &=&\frac{2\alpha \lambda \left( 1-\alpha \right) ^{3}}{3}-\frac{%
\left( 1-\alpha \right) ^{4}}{4} \\
\gamma _{4} &=&\left( 1-\alpha \right) ^{2}\left[ \alpha \lambda -\frac{%
1-\alpha }{3}-\frac{2\alpha \lambda \left( 1-\alpha \right) }{3}+\frac{%
\left( 1-\alpha \right) ^{2}}{4}\right]
\end{eqnarray*}%
and%
\begin{eqnarray*}
\mu _{1} &=&\frac{4}{3}\left( 1-\alpha \right) ^{3}\lambda ^{3}\left(
1-\lambda \left( 1-\alpha \right) \right) -\frac{1}{12}\left( \alpha
-2\lambda \left( 1-\alpha \right) \right) ^{2}\left[ \alpha \left( 3\alpha
-4\right) -4\lambda \left( 1-\alpha \right) ^{2}\left( 1-\lambda \right) %
\right] \\
\mu _{2} &=&\frac{4}{3}\left( 1-\alpha \right) ^{4}\lambda ^{4}+\frac{1}{12}%
\left( \alpha -2\lambda \left( 1-\alpha \right) \right) ^{2}\left[ \alpha
^{2}\left( 4\lambda ^{4}-4\lambda +3\right) +4\alpha \lambda \left(
1-2\lambda \right) +4\lambda ^{2}\right] \\
\mu _{3} &=&\frac{4}{3}\left( 1-\alpha \right) ^{3}\lambda ^{3}\left(
1-\lambda \left( 1-\alpha \right) \right) +\frac{1}{12}\left( \alpha
-2\lambda \left( 1-\alpha \right) \right) ^{2}\left[ \alpha \left( 3\alpha
-4\right) -4\lambda \left( 1-\alpha \right) ^{2}\left( 1-\lambda \right) %
\right] \\
\mu _{4} &=&\frac{4}{3}\left( 1-\alpha \right) ^{4}\lambda ^{4}-\frac{1}{12}%
\left( \alpha -2\lambda \left( 1-\alpha \right) \right) ^{2}\left[ \alpha
^{2}\left( 4\lambda ^{4}-4\lambda +3\right) +4\alpha \lambda \left(
1-2\lambda \right) +4\lambda ^{2}\right] .
\end{eqnarray*}
\end{theorem}

\begin{proof}
By using Lemma \ref{aaa}, properties of absolute value and using convexity
of $\left\vert f^{\prime \prime }\right\vert $ we have,%
\begin{eqnarray}
&&\left\vert 
\begin{array}{c}
\left( b-a\right) \left( \alpha -\frac{1}{2}\right) f^{\prime }\left( \left(
1-\alpha \right) b+\alpha a\right) -\frac{1}{b-a}\int_{a}^{b}f\left(
x\right) dx \\ 
\\ 
+\left( 1-\lambda \right) f\left( \left( 1-\alpha \right) b+\alpha a\right)
+\lambda \left( \alpha f\left( a\right) +\left( 1-\alpha \right) f\left(
b\right) \right)%
\end{array}%
\right\vert  \notag \\
&\leq &\frac{\left( b-a\right) ^{2}}{2}\left\{ \int_{0}^{1-\alpha
}t\left\vert 2\alpha \lambda -t\right\vert \left\vert f^{\prime \prime
}\left( tb+\left( 1-t\right) a\right) \right\vert dt\right.  \notag \\
&&\text{ \ \ \ \ \ \ \ \ \ \ \ }\left. +\int_{1-\alpha }^{1}\left(
1-t\right) \left\vert 1-2\lambda \left( 1-\alpha \right) -t\right\vert
\left\vert f^{\prime \prime }\left( tb+\left( 1-t\right) a\right)
\right\vert dt\right\}  \notag \\
&\leq &\frac{\left( b-a\right) ^{2}}{2}\left\{ \int_{0}^{1-\alpha
}t\left\vert 2\alpha \lambda -t\right\vert \left( t\left\vert f^{\prime
\prime }\left( b\right) \right\vert +\left( 1-t\right) \left\vert f^{\prime
\prime }\left( a\right) \right\vert \right) dt\right.  \notag \\
&&\text{\ \ \ \ \ \ \ \ \ \ \ \ }\left. +\int_{1-\alpha }^{1}\left(
1-t\right) \left\vert 1-2\lambda \left( 1-\alpha \right) -t\right\vert
\left( t\left\vert f^{\prime \prime }\left( b\right) \right\vert +\left(
1-t\right) \left\vert f^{\prime \prime }\left( a\right) \right\vert \right)
dt\right\}  \notag \\
&=&\frac{\left( b-a\right) ^{2}}{2}\left\{ \int_{0}^{1-\alpha }\left\vert
2\alpha \lambda -t\right\vert \left( t^{2}\left\vert f^{\prime \prime
}\left( b\right) \right\vert +t\left( 1-t\right) \left\vert f^{\prime \prime
}\left( a\right) \right\vert \right) dt\right.  \notag \\
&&\text{ \ \ \ \ \ \ \ \ \ \ \ }\left. +\int_{1-\alpha }^{1}\left\vert
1-2\lambda \left( 1-\alpha \right) -t\right\vert \left( t\left( 1-t\right)
\left\vert f^{\prime \prime }\left( b\right) \right\vert +\left( 1-t\right)
^{2}\left\vert f^{\prime \prime }\left( a\right) \right\vert \right)
dt\right\} .  \label{b}
\end{eqnarray}%
Hence by simple calculation%
\begin{gather}
\int_{0}^{1-\alpha }\left\vert 2\alpha \lambda -t\right\vert \left(
t^{2}\left\vert f^{\prime \prime }\left( b\right) \right\vert +t\left(
1-t\right) \left\vert f^{\prime \prime }\left( a\right) \right\vert \right)
dt  \label{b2} \\
=\left\{ 
\begin{array}{c}
\gamma _{1}\left\vert f^{\prime \prime }\left( b\right) \right\vert +\gamma
_{2}\left\vert f^{\prime \prime }\left( a\right) \right\vert ,\text{ \ \ \ \ 
}2\alpha \lambda \leq 1-\alpha \\ 
\gamma _{3}\left\vert f^{\prime \prime }\left( b\right) \right\vert +\gamma
_{4}\left\vert f^{\prime \prime }\left( a\right) \right\vert ,\text{ \ \ \ \ 
}2\alpha \lambda \geq 1-\alpha%
\end{array}%
\right. ,  \notag
\end{gather}%
\begin{eqnarray*}
\gamma _{1} &=&\frac{8}{3}\left( \alpha \lambda \right) ^{4}+\left( 1-\alpha
\right) ^{3}\left( \frac{1-\alpha }{4}-\frac{2\alpha \lambda }{3}\right) \\
\gamma _{2} &=&\frac{8}{3}\left( \alpha \lambda \right) ^{3}\left( 1-\alpha
\lambda \right) +\left( 1-\alpha \right) ^{2}\left[ \frac{1-\alpha }{3}%
-\alpha \lambda -\frac{\left( 1-\alpha \right) ^{2}}{4}+\frac{2\alpha
\lambda \left( 1-\alpha \right) }{3}\right] \\
\gamma _{3} &=&\frac{2\alpha \lambda \left( 1-\alpha \right) ^{3}}{3}-\frac{%
\left( 1-\alpha \right) ^{4}}{4} \\
\gamma _{4} &=&\left( 1-\alpha \right) ^{2}\left[ \alpha \lambda -\frac{%
1-\alpha }{3}-\frac{2\alpha \lambda \left( 1-\alpha \right) }{3}+\frac{%
\left( 1-\alpha \right) ^{2}}{4}\right]
\end{eqnarray*}%
and%
\begin{eqnarray}
&&\int_{1-\alpha }^{1}\left\vert 1-2\lambda \left( 1-\alpha \right)
-t\right\vert \left( t\left( 1-t\right) \left\vert f^{\prime \prime }\left(
b\right) \right\vert +\left( 1-t\right) ^{2}\left\vert f^{\prime \prime
}\left( a\right) \right\vert \right) dt  \label{b3} \\
&=&\left\{ 
\begin{array}{c}
\mu _{1}\left\vert f^{\prime \prime }\left( b\right) \right\vert +\mu
_{2}\left\vert f^{\prime \prime }\left( a\right) \right\vert ,\text{ \ \ \ \ 
}1-\alpha \leq 1-2\lambda \left( 1-\alpha \right) \\ 
\mu _{3}\left\vert f^{\prime \prime }\left( b\right) \right\vert +\mu
_{4}\left\vert f^{\prime \prime }\left( a\right) \right\vert ,\text{ \ \ \ \ 
}1-\alpha \geq 1-2\lambda \left( 1-\alpha \right)%
\end{array}%
\right. ,  \notag
\end{eqnarray}%
\begin{eqnarray*}
\mu _{1} &=&\frac{4}{3}\left( 1-\alpha \right) ^{3}\lambda ^{3}\left(
1-\lambda \left( 1-\alpha \right) \right) -\frac{1}{12}\left( \alpha
-2\lambda \left( 1-\alpha \right) \right) ^{2}\left[ \alpha \left( 3\alpha
-4\right) -4\lambda \left( 1-\alpha \right) ^{2}\left( 1-\lambda \right) %
\right] \\
\mu _{2} &=&\frac{4}{3}\left( 1-\alpha \right) ^{4}\lambda ^{4}+\frac{1}{12}%
\left( \alpha -2\lambda \left( 1-\alpha \right) \right) ^{2}\left[ \alpha
^{2}\left( 4\lambda ^{4}-4\lambda +3\right) +4\alpha \lambda \left(
1-2\lambda \right) +4\lambda ^{2}\right] \\
\mu _{3} &=&\frac{4}{3}\left( 1-\alpha \right) ^{3}\lambda ^{3}\left(
1-\lambda \left( 1-\alpha \right) \right) +\frac{1}{12}\left( \alpha
-2\lambda \left( 1-\alpha \right) \right) ^{2}\left[ \alpha \left( 3\alpha
-4\right) -4\lambda \left( 1-\alpha \right) ^{2}\left( 1-\lambda \right) %
\right] \\
\mu _{4} &=&\frac{4}{3}\left( 1-\alpha \right) ^{4}\lambda ^{4}-\frac{1}{12}%
\left( \alpha -2\lambda \left( 1-\alpha \right) \right) ^{2}\left[ \alpha
^{2}\left( 4\lambda ^{4}-4\lambda +3\right) +4\alpha \lambda \left(
1-2\lambda \right) +4\lambda ^{2}\right] .
\end{eqnarray*}%
Thus, using (\ref{b2}) and (\ref{b3}) in (\ref{b}), we obtain (\ref{b1}).
This completes the proof.
\end{proof}

\begin{theorem}
\label{hhh}Let $f:I\subset 
\mathbb{R}
\rightarrow 
\mathbb{R}
$ be a twice differentiable mapping on $I^{\circ }$ such that $f,f^{\prime
},f^{\prime \prime }\in L\left[ a,b\right] $, where $a,b\in I$ with $a<b$
and $\alpha ,\lambda \in \left[ 0,1\right] $. If $\left\vert f^{\prime
\prime }\right\vert ^{q}$ is convex on $\left[ a,b\right] ,$ for $q\geq 1,$
the following inequalities hold:%
\begin{gather}
\left\vert 
\begin{array}{c}
\left( b-a\right) \left( \alpha -\frac{1}{2}\right) f^{\prime }\left( \left(
1-\alpha \right) b+\alpha a\right) -\frac{1}{b-a}\int_{a}^{b}f\left(
x\right) dx \\ 
\\ 
+\left( 1-\lambda \right) f\left( \left( 1-\alpha \right) b+\alpha a\right)
+\lambda \left( \alpha f\left( a\right) +\left( 1-\alpha \right) f\left(
b\right) \right)%
\end{array}%
\right\vert  \label{k} \\
\notag \\
\leq \left\{ 
\begin{array}{c}
\frac{\left( b-a\right) ^{2}}{2}\left[ 
\begin{array}{c}
\tau _{1}^{1-\frac{1}{q}}\left( \gamma _{1}\left\vert {\small f}^{{\small %
\prime \prime }}\left( {\small b}\right) \right\vert ^{q}+\gamma
_{2}\left\vert {\small f}^{{\small \prime \prime }}\left( {\small a}\right)
\right\vert ^{q}\right) ^{\frac{1}{q}} \\ 
+z_{1}^{1-\frac{1}{q}}\left( \mu _{1}\left\vert {\small f}^{{\small \prime
\prime }}\left( {\small b}\right) \right\vert ^{q}+\mu _{2}\left\vert 
{\small f}^{{\small \prime \prime }}\left( {\small a}\right) \right\vert
^{q}\right) ^{\frac{1}{q}}%
\end{array}%
\right] ,\text{ \ \ \ \ \ \ }{\small 2\alpha \lambda \leq 1-\alpha \leq
1-2\lambda }\left( {\small 1-\alpha }\right) \\ 
\\ 
\frac{\left( b-a\right) ^{2}}{2}\left[ 
\begin{array}{c}
\tau _{1}^{1-\frac{1}{q}}\left( \gamma _{1}\left\vert {\small f}^{{\small %
\prime \prime }}\left( {\small b}\right) \right\vert ^{q}+\gamma
_{2}\left\vert {\small f}^{{\small \prime \prime }}\left( {\small a}\right)
\right\vert ^{q}\right) ^{\frac{1}{q}} \\ 
+z_{2}^{1-\frac{1}{q}}\left( \mu _{3}\left\vert {\small f}^{{\small \prime
\prime }}\left( {\small b}\right) \right\vert ^{q}+\mu _{4}\left\vert 
{\small f}^{{\small \prime \prime }}\left( {\small a}\right) \right\vert
^{q}\right) ^{\frac{1}{q}}%
\end{array}%
\right] ,\text{ }{\small 1-\alpha \geq }\max \left\{ {\small 2\alpha \lambda
,1-2\lambda }\left( {\small 1-\alpha }\right) \right\} \\ 
\\ 
\frac{\left( b-a\right) ^{2}}{2}\left[ 
\begin{array}{c}
\tau _{2}^{1-\frac{1}{q}}\left( \gamma _{3}\left\vert {\small f}^{{\small %
\prime \prime }}\left( {\small b}\right) \right\vert ^{q}+\gamma
_{4}\left\vert {\small f}^{{\small \prime \prime }}\left( {\small a}\right)
\right\vert ^{q}\right) ^{\frac{1}{q}} \\ 
+z_{1}^{1-\frac{1}{q}}\left( \mu _{1}\left\vert {\small f}^{{\small \prime
\prime }}\left( {\small b}\right) \right\vert ^{q}+\mu _{2}\left\vert 
{\small f}^{{\small \prime \prime }}\left( {\small a}\right) \right\vert
^{q}\right) ^{\frac{1}{q}}%
\end{array}%
\right] ,\text{\ }{\small 1-\alpha \leq }\min \left\{ {\small 2\alpha
\lambda ,1-2\lambda }\left( {\small 1-\alpha }\right) \right\} \\ 
\\ 
\frac{\left( b-a\right) ^{2}}{2}\left[ 
\begin{array}{c}
\tau _{2}^{1-\frac{1}{q}}\left( \gamma _{3}\left\vert {\small f}^{{\small %
\prime \prime }}\left( {\small b}\right) \right\vert ^{q}+\gamma
_{4}\left\vert {\small f}^{{\small \prime \prime }}\left( {\small a}\right)
\right\vert ^{q}\right) ^{\frac{1}{q}} \\ 
+z_{2}^{1-\frac{1}{q}}\left( \mu _{3}\left\vert {\small f}^{{\small \prime
\prime }}\left( {\small b}\right) \right\vert ^{q}+\mu _{4}\left\vert 
{\small f}^{{\small \prime \prime }}\left( {\small a}\right) \right\vert
^{q}\right) ^{\frac{1}{q}}%
\end{array}%
\right] ,\text{ \ \ \ \ \ }{\small 1-2\lambda }\left( {\small 1-\alpha }%
\right) {\small \leq 1-\alpha \leq 2\alpha \lambda },%
\end{array}%
\right.  \notag
\end{gather}%
where%
\begin{eqnarray*}
\tau _{1} &=&\frac{8}{3}\left( \alpha \lambda \right) ^{3}+\left( 1-\alpha
\right) ^{2}\left( \frac{1-\alpha }{3}-\alpha \lambda \right) \\
\tau _{2} &=&\left( 1-\alpha \right) ^{2}\left( \alpha \lambda -\frac{%
1-\alpha }{3}\right)
\end{eqnarray*}%
and%
\begin{eqnarray*}
z_{1} &=&\frac{4}{3}\left( 1-\alpha \right) ^{3}\lambda ^{3}+\frac{1}{3}%
\left( \alpha -2\lambda \left( 1-\alpha \right) \right) ^{2}\left( \alpha
\left( 1-\lambda \right) +\lambda \right) \\
z_{2} &=&\frac{4}{3}\left( 1-\alpha \right) ^{3}\lambda ^{3}-\frac{1}{3}%
\left( \alpha -2\lambda \left( 1-\alpha \right) \right) ^{2}\left( \alpha
\left( 1-\lambda \right) +\lambda \right)
\end{eqnarray*}%
$\gamma _{i}$ and $\mu _{i}$ $(i=1,2,3,4)$ are defined as in Theorem \ref%
{ggg}.
\end{theorem}

\begin{proof}
Suppose that $q\geq 1.$ From Lemma \ref{aaa} and using well known power mean
inequality, we have%
\begin{eqnarray*}
&&\left\vert 
\begin{array}{c}
\left( b-a\right) \left( \alpha -\frac{1}{2}\right) f^{\prime }\left( \left(
1-\alpha \right) b+\alpha a\right) -\frac{1}{b-a}\int_{a}^{b}f\left(
x\right) dx \\ 
\\ 
+\left( 1-\lambda \right) f\left( \left( 1-\alpha \right) b+\alpha a\right)
+\lambda \left( \alpha f\left( a\right) +\left( 1-\alpha \right) f\left(
b\right) \right)%
\end{array}%
\right\vert \\
&& \\
&\leq &\frac{\left( b-a\right) ^{2}}{2}\left\{ \int_{0}^{1-\alpha
}\left\vert 2\alpha \lambda t-t^{2}\right\vert \left\vert f^{\prime \prime
}\left( tb+\left( 1-t\right) a\right) \right\vert dt\right. \\
&&\text{\ }\left. +\int_{1-\alpha }^{1}\left\vert \left( 1-t\right) \left(
1-2\lambda \left( 1-\alpha \right) -t\right) \right\vert \left\vert
f^{\prime \prime }\left( tb+\left( 1-t\right) a\right) \right\vert dt\right\}
\\
&& \\
&\leq &\frac{\left( b-a\right) ^{2}}{2}\left[ \left( \int_{0}^{1-\alpha
}\left\vert 2\alpha \lambda t-t^{2}\right\vert dt\right) ^{1-\frac{1}{q}%
}\right. \\
&&\text{ \ \ \ \ \ \ \ \ \ \ \ }\times \left( \int_{0}^{1-\alpha }\left\vert
2\alpha \lambda t-t^{2}\right\vert \left\vert f^{\prime \prime }\left(
tb+\left( 1-t\right) a\right) \right\vert ^{q}dt\right) ^{\frac{1}{q}} \\
&&\text{ \ \ \ \ \ \ \ \ \ \ \ }+\left( \int_{1-\alpha }^{1}\left\vert
\left( 1-t\right) \left( 1-2\lambda \left( 1-\alpha \right) -t\right)
\right\vert dt\right) ^{1-\frac{1}{q}} \\
&&\text{ \ \ \ \ \ \ \ \ \ \ \ }\left. \times \left( \int_{1-\alpha
}^{1}\left\vert \left( 1-t\right) \left( 1-2\lambda \left( 1-\alpha \right)
-t\right) \right\vert \left\vert f^{\prime \prime }\left( tb+\left(
1-t\right) a\right) \right\vert ^{q}dt\right) ^{\frac{1}{q}}\right] .
\end{eqnarray*}%
By using convexity of $\left\vert f^{\prime \prime }\right\vert ^{q}$ we
know that%
\begin{equation*}
\left\vert f^{\prime \prime }(ta+(1-t)b)\right\vert ^{q}\leq t\left\vert
f^{\prime \prime }(a)\right\vert ^{q}+(1-t)\left\vert f^{\prime \prime
}(b)\right\vert ^{q}.
\end{equation*}%
So we have 
\begin{eqnarray}
&&\left\vert 
\begin{array}{c}
\left( b-a\right) \left( \alpha -\frac{1}{2}\right) f^{\prime }\left( \left(
1-\alpha \right) b+\alpha a\right) -\frac{1}{b-a}\int_{a}^{b}f\left(
x\right) dx \\ 
\\ 
+\left( 1-\lambda \right) f\left( \left( 1-\alpha \right) b+\alpha a\right)
+\lambda \left( \alpha f\left( a\right) +\left( 1-\alpha \right) f\left(
b\right) \right)%
\end{array}%
\right\vert  \notag \\
&\leq &\frac{\left( b-a\right) ^{2}}{2}\left[ \left( \int_{0}^{1-\alpha
}t\left\vert 2\alpha \lambda -t\right\vert dt\right) ^{1-\frac{1}{q}}\right.
\notag \\
&&\text{ \ \ \ \ \ \ \ \ \ \ \ }\times \left( \int_{0}^{1-\alpha }\left\vert
2\alpha \lambda -t\right\vert \left( t^{2}\left\vert f^{\prime \prime
}\left( b\right) \right\vert ^{q}+t\left( 1-t\right) \left\vert f^{\prime
\prime }\left( a\right) \right\vert ^{q}\right) dt\right) ^{\frac{1}{q}} 
\notag \\
&&\text{ \ \ \ \ \ \ \ \ \ \ \ }+\left( \int_{1-\alpha }^{1}\left(
1-t\right) \left\vert \left( 1-2\lambda \left( 1-\alpha \right) -t\right)
\right\vert dt\right) ^{1-\frac{1}{q}}  \label{k1} \\
&&\text{ \ \ \ \ \ \ \ \ \ \ \ }\left. \times \left( \int_{1-\alpha
}^{1}\left\vert \left( 1-2\lambda \left( 1-\alpha \right) -t\right)
\right\vert \left( t\left( 1-t\right) \left\vert f^{\prime \prime }\left(
b\right) \right\vert ^{q}+\left( 1-t\right) ^{2}\left\vert f^{\prime \prime
}\left( a\right) \right\vert ^{q}\right) dt\right) ^{\frac{1}{q}}\right] . 
\notag
\end{eqnarray}%
Hence, by simple computation%
\begin{equation}
\int_{0}^{1-\alpha }t\left\vert 2\alpha \lambda -t\right\vert dt=\left\{ 
\begin{array}{c}
\tau _{1},\text{ \ \ }2\alpha \lambda \leq 1-\alpha \\ 
\\ 
\tau _{2},\text{ \ \ }2\alpha \lambda \geq 1-\alpha%
\end{array}%
\right. ,  \label{k2}
\end{equation}%
\begin{eqnarray*}
\tau _{1} &=&\frac{8}{3}\left( \alpha \lambda \right) ^{3}+\left( 1-\alpha
\right) ^{2}\left( \frac{1-\alpha }{3}-\alpha \lambda \right) \\
\tau _{2} &=&\left( 1-\alpha \right) ^{2}\left( \alpha \lambda -\frac{%
1-\alpha }{3}\right)
\end{eqnarray*}%
\begin{equation}
\int_{1-\alpha }^{1}\left( 1-t\right) \left\vert \left( 1-2\lambda \left(
1-\alpha \right) -t\right) \right\vert dt=\left\{ 
\begin{array}{c}
z_{1},\text{ \ \ }1-2\lambda \left( 1-\alpha \right) \geq 1-\alpha \\ 
\\ 
z_{2},\text{ \ \ }1-2\lambda \left( 1-\alpha \right) \leq 1-\alpha%
\end{array}%
\right.  \label{k3}
\end{equation}%
\begin{eqnarray*}
z_{1} &=&\frac{4}{3}\left( 1-\alpha \right) ^{3}\lambda ^{3}+\frac{1}{3}%
\left( \alpha -2\lambda \left( 1-\alpha \right) \right) ^{2}\left( \alpha
\left( 1-\lambda \right) +\lambda \right) \\
z_{2} &=&\frac{4}{3}\left( 1-\alpha \right) ^{3}\lambda ^{3}-\frac{1}{3}%
\left( \alpha -2\lambda \left( 1-\alpha \right) \right) ^{2}\left( \alpha
\left( 1-\lambda \right) +\lambda \right)
\end{eqnarray*}%
\begin{eqnarray}
&&\int_{0}^{1-\alpha }\left\vert 2\alpha \lambda -t\right\vert \left(
t^{2}\left\vert f^{\prime \prime }\left( b\right) \right\vert ^{q}+t\left(
1-t\right) \left\vert f^{\prime \prime }\left( a\right) \right\vert
^{q}\right) dt  \label{k4} \\
&&  \notag \\
&=&\left\{ 
\begin{array}{c}
\gamma _{1}\left\vert f^{\prime \prime }\left( b\right) \right\vert
^{q}+\gamma _{2}\left\vert f^{\prime \prime }\left( a\right) \right\vert
^{q},\text{ \ \ \ \ }2\alpha \lambda \leq 1-\alpha \\ 
\\ 
\gamma _{3}\left\vert f^{\prime \prime }\left( b\right) \right\vert
^{q}+\gamma _{4}\left\vert f^{\prime \prime }\left( a\right) \right\vert
^{q},\text{ \ \ \ \ }2\alpha \lambda \geq 1-\alpha%
\end{array}%
\right.  \notag
\end{eqnarray}%
and $\gamma _{1},$ $\gamma _{2},$ $\gamma _{3},$ $\gamma _{4}$ are defined
as in Theorem \ref{ggg}.%
\begin{eqnarray}
&&\int_{1-\alpha }^{1}\left\vert \left( 1-2\lambda \left( 1-\alpha \right)
-t\right) \right\vert \left( t\left( 1-t\right) \left\vert f^{\prime \prime
}\left( b\right) \right\vert ^{q}+\left( 1-t\right) ^{2}\left\vert f^{\prime
\prime }\left( a\right) \right\vert ^{q}\right) dt  \label{k5} \\
&&  \notag \\
&=&\left\{ 
\begin{array}{c}
\mu _{1}\left\vert f^{\prime \prime }\left( b\right) \right\vert ^{q}+\mu
_{2}\left\vert f^{\prime \prime }\left( a\right) \right\vert ^{q},\text{ \ \
\ \ }1-\alpha \leq 1-2\lambda \left( 1-\alpha \right) \\ 
\\ 
\mu _{3}\left\vert f^{\prime \prime }\left( b\right) \right\vert ^{q}+\mu
_{4}\left\vert f^{\prime \prime }\left( a\right) \right\vert ^{q},\text{ \ \
\ \ }1-\alpha \geq 1-2\lambda \left( 1-\alpha \right)%
\end{array}%
\right.  \notag
\end{eqnarray}%
and $\mu _{1},$ $\mu _{2},$ $\mu _{3},$ $\mu _{4}$ are defined as in Theorem %
\ref{ggg}. Thus, using (\ref{k2})-(\ref{k5}) in (\ref{k1}), we get (\ref{k}%
). So the proof is completed.
\end{proof}

\begin{theorem}
\label{ddd}Let $f:I\subset 
\mathbb{R}
\rightarrow 
\mathbb{R}
$ be a twice differentiable mapping on $I^{\circ }$ such that $f,f^{\prime
},f^{\prime \prime }\in L\left[ a,b\right] $, where $a,b\in I$ with $a<b$
and $\alpha ,\lambda \in \left[ 0,1\right] $. If $\left\vert f^{\prime
\prime }\right\vert ^{q}$ is convex on $\left[ a,b\right] ,$ for $p,q\geq 1,$
$\frac{1}{p}+\frac{1}{q}=1,$ the following inequalities hold:%
\begin{eqnarray*}
&&\left\vert 
\begin{array}{c}
\left( b-a\right) \left( \alpha -\frac{1}{2}\right) f^{\prime }\left( \left(
1-\alpha \right) b+\alpha a\right) -\frac{1}{b-a}\int_{a}^{b}f\left(
x\right) dx \\ 
\\ 
+\left( 1-\lambda \right) f\left( \left( 1-\alpha \right) b+\alpha a\right)
+\lambda \left( \alpha f\left( a\right) +\left( 1-\alpha \right) f\left(
b\right) \right)%
\end{array}%
\right\vert \\
&& \\
&\leq &\left\{ 
\begin{array}{c}
\frac{\left( b-a\right) ^{2}}{2}\left\{ 
\begin{array}{c}
\varphi _{1}^{\frac{1}{p}}\left[ \varepsilon _{1}\left\vert {\small f}^{%
{\small \prime \prime }}\left( {\small b}\right) \right\vert ^{q}+\beta
\left( 1-\alpha ;q+1,2\right) \left\vert {\small f}^{{\small \prime \prime }%
}\left( {\small a}\right) \right\vert ^{q}\right] ^{\frac{1}{q}} \\ 
+\psi _{1}^{\frac{1}{p}}\left[ \beta \left( \alpha ;q+1,2\right) \left\vert 
{\small f}^{{\small \prime \prime }}\left( {\small b}\right) \right\vert
^{q}+\varepsilon _{2}\left\vert {\small f}^{{\small \prime \prime }}\left( 
{\small a}\right) \right\vert ^{q}\right] ^{\frac{1}{q}}%
\end{array}%
\right\} ,\text{ \ \ \ \ \ \ }{\small 2\alpha \lambda \leq 1-\alpha \leq
1-2\lambda }\left( {\small 1-\alpha }\right) \\ 
\\ 
\frac{\left( b-a\right) ^{2}}{2}\left\{ 
\begin{array}{c}
\varphi _{1}^{\frac{1}{p}}\left[ \varepsilon _{1}\left\vert {\small f}^{%
{\small \prime \prime }}\left( {\small b}\right) \right\vert ^{q}+\beta
\left( 1-\alpha ;q+1,2\right) \left\vert {\small f}^{{\small \prime \prime }%
}\left( {\small a}\right) \right\vert ^{q}\right] ^{\frac{1}{q}} \\ 
+\psi _{2}^{\frac{1}{p}}\left[ \beta \left( \alpha ;q+1,2\right) \left\vert 
{\small f}^{{\small \prime \prime }}\left( {\small b}\right) \right\vert
^{q}+\varepsilon _{2}\left\vert {\small f}^{{\small \prime \prime }}\left( 
{\small a}\right) \right\vert ^{q}\right] ^{\frac{1}{q}}%
\end{array}%
\right\} ,\text{ }{\small 1-\alpha \geq }\max \left\{ {\small 2\alpha
\lambda ,1-2\lambda }\left( {\small 1-\alpha }\right) \right\} \\ 
\\ 
\frac{\left( b-a\right) ^{2}}{2}\left\{ 
\begin{array}{c}
\varphi _{2}^{\frac{1}{p}}\left[ \varepsilon _{1}\left\vert {\small f}^{%
{\small \prime \prime }}\left( {\small b}\right) \right\vert ^{q}+\beta
\left( 1-\alpha ;q+1,2\right) \left\vert {\small f}^{{\small \prime \prime }%
}\left( {\small a}\right) \right\vert ^{q}\right] ^{\frac{1}{q}} \\ 
+\psi _{1}^{\frac{1}{p}}\left[ \beta \left( \alpha ;q+1,2\right) \left\vert 
{\small f}^{{\small \prime \prime }}\left( {\small b}\right) \right\vert
^{q}+\varepsilon _{2}\left\vert {\small f}^{{\small \prime \prime }}\left( 
{\small a}\right) \right\vert ^{q}\right] ^{\frac{1}{q}}%
\end{array}%
\right\} ,\text{\ }{\small 1-\alpha \leq }\min \left\{ {\small 2\alpha
\lambda ,1-2\lambda }\left( {\small 1-\alpha }\right) \right\} \\ 
\\ 
\frac{\left( b-a\right) ^{2}}{2}\left\{ 
\begin{array}{c}
\varphi _{2}^{\frac{1}{p}}\left[ \varepsilon _{1}\left\vert {\small f}^{%
{\small \prime \prime }}\left( {\small b}\right) \right\vert ^{q}+\beta
\left( 1-\alpha ;q+1,2\right) \left\vert {\small f}^{{\small \prime \prime }%
}\left( {\small a}\right) \right\vert ^{q}\right] ^{\frac{1}{q}} \\ 
+\psi _{2}^{\frac{1}{p}}\left[ \beta \left( \alpha ;q+1,2\right) \left\vert 
{\small f}^{{\small \prime \prime }}\left( {\small b}\right) \right\vert
^{q}+\varepsilon _{2}\left\vert {\small f}^{{\small \prime \prime }}\left( 
{\small a}\right) \right\vert ^{q}\right] ^{\frac{1}{q}}%
\end{array}%
\right\} ,\text{ \ \ \ \ \ }{\small 1-2\lambda }\left( {\small 1-\alpha }%
\right) {\small \leq 1-\alpha \leq 2\alpha \lambda },%
\end{array}%
\right.
\end{eqnarray*}%
where%
\begin{eqnarray*}
\varphi _{1} &=&\frac{\left( 2\alpha \lambda \right) ^{p+1}+\left( 1-\alpha
\left( 1+2\lambda \right) \right) ^{p+1}}{p+1} \\
\varphi _{2} &=&\frac{\left( 2\alpha \lambda \right) ^{p+1}-\left( -1\right)
^{p+1}\left( 1-\alpha \left( 1+2\lambda \right) \right) ^{p+1}}{p+1}
\end{eqnarray*}%
\begin{eqnarray*}
\psi _{1} &=&\frac{\left( 2\lambda \left( 1-\alpha \right) \right)
^{p+1}+\left( \alpha -2\lambda \left( 1-\alpha \right) \right) ^{p+1}}{p+1}
\\
\psi _{2} &=&\frac{\left( 2\lambda \left( 1-\alpha \right) \right)
^{p+1}-\left( -1\right) ^{p+1}\left( \alpha -2\lambda \left( 1-\alpha
\right) \right) ^{p+1}}{p+1}
\end{eqnarray*}%
\begin{eqnarray*}
\varepsilon _{1} &=&\frac{\left( 1-\alpha \right) ^{q+2}}{q+2} \\
\varepsilon _{2} &=&\frac{\alpha ^{q+2}}{q+2}
\end{eqnarray*}%
and $\beta $ is incomplete Beta function.
\end{theorem}

\begin{proof}
From Lemma \ref{aaa} and properties of absolute value and using H\"{o}lder
inequality, for $p,q\geq 1,$ $\frac{1}{p}+\frac{1}{q}=1$ we have%
\begin{eqnarray*}
I &=&\left\vert 
\begin{array}{c}
\left( b-a\right) \left( \alpha -\frac{1}{2}\right) f^{\prime }\left( \left(
1-\alpha \right) b+\alpha a\right) -\frac{1}{b-a}\int_{a}^{b}f\left(
x\right) dx \\ 
\\ 
+\left( 1-\lambda \right) f\left( \left( 1-\alpha \right) b+\alpha a\right)
+\lambda \left( \alpha f\left( a\right) +\left( 1-\alpha \right) f\left(
b\right) \right)%
\end{array}%
\right\vert \\
&& \\
&\leq &\frac{\left( b-a\right) ^{2}}{2}\left\{ \int_{0}^{1-\alpha
}\left\vert 2\alpha \lambda t-t^{2}\right\vert \left\vert f^{\prime \prime
}\left( tb+\left( 1-t\right) a\right) \right\vert dt\right. \\
&&\text{\ }\left. +\int_{1-\alpha }^{1}\left\vert \left( 1-t\right) \left(
1-2\lambda \left( 1-\alpha \right) -t\right) \right\vert \left\vert
f^{\prime \prime }\left( tb+\left( 1-t\right) a\right) \right\vert dt\right\}
\\
&& \\
&=&\frac{\left( b-a\right) ^{2}}{2}\left[ \left( \int_{0}^{1-\alpha
}\left\vert 2\alpha \lambda -t\right\vert \left[ t\left\vert f^{\prime
\prime }\left( tb+\left( 1-t\right) a\right) \right\vert \right] dt\right)
\right. \\
&&\text{ \ \ \ \ \ \ \ \ \ \ \ }\left. +\left( \int_{1-\alpha
}^{1}\left\vert \left( 1-2\lambda \left( 1-\alpha \right) -t\right)
\right\vert \left[ \left( 1-t\right) \left\vert f^{\prime \prime }\left(
tb+\left( 1-t\right) a\right) \right\vert \right] dt\right) \right] \\
&& \\
&\leq &\frac{\left( b-a\right) ^{2}}{2}\left[ \left( \int_{0}^{1-\alpha
}\left\vert 2\alpha \lambda -t\right\vert ^{p}dt\right) ^{\frac{1}{p}}\left(
\int_{0}^{1-\alpha }t^{q}\left\vert f^{\prime \prime }\left( tb+\left(
1-t\right) a\right) \right\vert ^{q}dt\right) ^{\frac{1}{q}}\right. \\
&&\left. +\left( \int_{1-\alpha }^{1}\left\vert \left( 1-2\lambda \left(
1-\alpha \right) -t\right) \right\vert ^{p}dt\right) ^{\frac{1}{p}}\left(
\int_{1-\alpha }^{1}\left( 1-t\right) ^{q}\left\vert f^{\prime \prime
}\left( tb+\left( 1-t\right) a\right) \right\vert ^{q}dt\right) ^{\frac{1}{q}%
}\right]
\end{eqnarray*}%
Since $\left\vert f^{\prime \prime }\right\vert ^{q}$ is convex on $\left[
a,b\right] $ we can write%
\begin{eqnarray}
I &\leq &\frac{\left( b-a\right) ^{2}}{2}\left[ 
\begin{array}{c}
\left( \int_{0}^{1-\alpha }\left\vert 2\alpha \lambda -t\right\vert
^{p}dt\right) ^{\frac{1}{p}} \\ 
\left( \int_{0}^{1-\alpha }t^{q}\left( t\left\vert f^{\prime \prime }\left(
b\right) \right\vert ^{q}+\left( 1-t\right) \left\vert f^{\prime \prime
}\left( a\right) \right\vert ^{q}\right) dt\right) ^{\frac{1}{q}}%
\end{array}%
\right.  \notag \\
&&\text{ \ \ \ \ \ \ \ \ \ \ \ \ \ \ \ }\left. 
\begin{array}{c}
+\left( \int_{1-\alpha }^{1}\left\vert \left( 1-2\lambda \left( 1-\alpha
\right) -t\right) \right\vert ^{p}dt\right) ^{\frac{1}{p}} \\ 
\left( \int_{1-\alpha }^{1}\left( 1-t\right) ^{q}\left( t\left\vert
f^{\prime \prime }\left( b\right) \right\vert ^{q}+\left( 1-t\right)
\left\vert f^{\prime \prime }\left( a\right) \right\vert ^{q}\right)
dt\right) ^{\frac{1}{q}}%
\end{array}%
\right]  \notag \\
&&  \notag \\
&=&\frac{\left( b-a\right) ^{2}}{2}\left[ 
\begin{array}{c}
\left( \int_{0}^{1-\alpha }\left\vert 2\alpha \lambda -t\right\vert
^{p}dt\right) ^{\frac{1}{p}} \\ 
\left( \left\vert f^{\prime \prime }\left( b\right) \right\vert
^{q}\int_{0}^{1-\alpha }t^{q+1}dt+\left\vert f^{\prime \prime }\left(
a\right) \right\vert ^{q}\int_{0}^{1-\alpha }t^{q}\left( 1-t\right)
dt\right) ^{\frac{1}{q}}%
\end{array}%
\right.  \notag \\
&&\text{ \ \ \ \ \ \ \ \ \ \ \ }\left. 
\begin{array}{c}
+\left( \int_{1-\alpha }^{1}\left\vert \left( 1-2\lambda \left( 1-\alpha
\right) -t\right) \right\vert ^{p}dt\right) ^{\frac{1}{p}} \\ 
\left( \left\vert f^{\prime \prime }\left( b\right) \right\vert
^{q}\int_{1-\alpha }^{1}t\left( 1-t\right) ^{q}dt+\left\vert f^{\prime
\prime }\left( a\right) \right\vert ^{q}\int_{1-\alpha }^{1}\left(
1-t\right) ^{q+1}dt\right) ^{\frac{1}{q}}%
\end{array}%
\right] .  \label{n1}
\end{eqnarray}

Making use of necessary computation%
\begin{equation}
\int_{0}^{1-\alpha }\left\vert 2\alpha \lambda -t\right\vert ^{p}dt=\left\{ 
\begin{array}{c}
\varphi _{1},\text{ \ \ \ \ }2\alpha \lambda \leq 1-\alpha \\ 
\\ 
\varphi _{2},\text{ \ \ \ \ }2\alpha \lambda \geq 1-\alpha%
\end{array}%
\right.  \label{n2}
\end{equation}%
\begin{eqnarray*}
\varphi _{1} &=&\frac{\left( 2\alpha \lambda \right) ^{p+1}+\left( 1-\alpha
\left( 1+2\lambda \right) \right) ^{p+1}}{p+1} \\
\varphi _{2} &=&\frac{\left( 2\alpha \lambda \right) ^{p+1}-\left( -1\right)
^{p+1}\left( 1-\alpha \left( 1+2\lambda \right) \right) ^{p+1}}{p+1}
\end{eqnarray*}%
\begin{equation}
\int_{1-\alpha }^{1}\left\vert \left( 1-2\lambda \left( 1-\alpha \right)
-t\right) \right\vert ^{p}dt=\left\{ 
\begin{array}{c}
\psi _{1},\text{ \ \ \ \ }1-\alpha \leq 1-2\lambda \left( 1-\alpha \right)
\\ 
\\ 
\psi _{2},\text{ \ \ \ \ }1-\alpha \geq 1-2\lambda \left( 1-\alpha \right)%
\end{array}%
\right.  \label{n3}
\end{equation}%
\begin{eqnarray*}
\psi _{1} &=&\frac{\left( 2\lambda \left( 1-\alpha \right) \right)
^{p+1}+\left( \alpha -2\lambda \left( 1-\alpha \right) \right) ^{p+1}}{p+1}
\\
\psi _{2} &=&\frac{\left( 2\lambda \left( 1-\alpha \right) \right)
^{p+1}-\left( -1\right) ^{p+1}\left( \alpha -2\lambda \left( 1-\alpha
\right) \right) ^{p+1}}{p+1}
\end{eqnarray*}%
\begin{eqnarray}
\int_{0}^{1-\alpha }t^{q+1}dt &=&\frac{\left( 1-\alpha \right) ^{q+2}}{q+2}
\label{n4} \\
\int_{1-\alpha }^{1}\left( 1-t\right) ^{q+1}dt &=&\frac{\alpha ^{q+2}}{q+2} 
\notag
\end{eqnarray}%
and%
\begin{eqnarray}
\int_{0}^{1-\alpha }t^{q}\left( 1-t\right) dt &=&\beta \left( 1-\alpha
;q+1,2\right)  \label{n5} \\
\int_{1-\alpha }^{1}t\left( 1-t\right) ^{q}dt &=&\beta \left( \alpha
;q+1,2\right) .  \notag
\end{eqnarray}%
where $\beta $ is incomplete Beta function. By using (\ref{n2})-(\ref{n5})
in (\ref{n1}), we get the desired result.
\end{proof}

\begin{remark}
In our main results, if we choose $\alpha =\frac{1}{2},$ then under
attendant assumptions, Lemma \ref{aaa}, Theorem \ref{ggg} and Theorem \ref%
{hhh} reduces to Lemma \ref{s1}, Theorem \ref{s2} and Theorem \ref{s3} in 
\cite{ana}, respectively.
\end{remark}

\begin{remark}
Under the assumptions of Lemma \ref{aaa}, by integrating both sides respect
to $\alpha $ over $\left[ 0,1\right] $ we get%
\begin{eqnarray*}
&&\left( \lambda -1\right) \left( \frac{f\left( a\right) +f\left( b\right) }{%
2}-\frac{1}{b-a}\int_{a}^{b}f\left( x\right) dx\right) \\
&=&\frac{\left( b-a\right) ^{2}}{2}\int_{0}^{1}\int_{0}^{1}k\left( \alpha
,t\right) f^{\prime \prime }\left( tb+\left( 1-t\right) a\right) dtd\alpha
\end{eqnarray*}%
where%
\begin{equation*}
k\left( t\right) =\left\{ 
\begin{array}{c}
\text{ \ \ \ \ \ \ \ \ \ }2\alpha \lambda t-t^{2}\text{, \ \ \ \ \ \ \ \ \ \
\ \ \ }0\leq t\leq 1-\alpha \\ 
\left( 1-t\right) \left( t-1+2\lambda \left( 1-\alpha \right) \right) \text{%
, \ \ \ }1-\alpha \leq t\leq 1.%
\end{array}%
\right.
\end{equation*}
\end{remark}

\section{Applications to Quadrature Formulas}

In this section we point out some particular inequalities which generalize
some classical results such as : trapezoid inequality, Simpson's inequality,
midpoint inequality.

\begin{proposition}[Midpoint inequality]
\label{pro1}Under the assumptions of Theorem \ref{ddd} with $\alpha =\frac{1%
}{2},$ $\lambda =0$, we get the following inequality:%
\begin{eqnarray*}
&&\left\vert f\left( \frac{a+b}{2}\right) -\frac{1}{b-a}\int_{a}^{b}f\left(
x\right) dx\right\vert \\
&\leq &\left( b-a\right) ^{2}\left( \frac{1}{2^{2p+1}\left( p+1\right) }%
\right) ^{\frac{1}{p}}\left\{ \left[ \frac{1}{2^{q+2}\left( q+2\right) }%
\left\vert f^{\prime \prime }\left( b\right) \right\vert ^{q}+\beta \left( 
\frac{1}{2};q+1,2\right) \left\vert f^{\prime \prime }\left( a\right)
\right\vert ^{q}\right] ^{\frac{1}{q}}\right. \\
&&\text{ \ \ \ \ \ \ \ \ \ \ \ \ \ \ \ \ \ \ \ \ \ \ \ \ \ \ \ \ \ \ \ }%
\left. +\left[ \beta \left( \frac{1}{2};q+1,2\right) \left\vert f^{\prime
\prime }\left( b\right) \right\vert ^{q}+\frac{1}{2^{q+2}\left( q+2\right) }%
\left\vert f^{\prime \prime }\left( a\right) \right\vert ^{q}\right] ^{\frac{%
1}{q}}\right\} .
\end{eqnarray*}
\end{proposition}

\begin{proposition}[Trapezoid inequality]
\label{pro2}Under the assumptions of Theorem \ref{ddd} with $\alpha =\frac{1%
}{2},$ $\lambda =1$, we get the following inequality:%
\begin{eqnarray*}
&&\left\vert \frac{f\left( a\right) +f\left( b\right) }{2}-\frac{1}{b-a}%
\int_{a}^{b}f\left( x\right) dx\right\vert \\
&\leq &\left( b-a\right) ^{2}\left( \frac{2^{p+1}-1}{2^{2p+1}\left(
p+1\right) }\right) ^{\frac{1}{p}}\left\{ \left[ \frac{1}{2^{q+2}\left(
q+2\right) }\left\vert f^{\prime \prime }\left( b\right) \right\vert
^{q}+\beta \left( \frac{1}{2};q+1,2\right) \left\vert f^{\prime \prime
}\left( a\right) \right\vert ^{q}\right] ^{\frac{1}{q}}\right. \\
&&\text{ \ \ \ \ \ \ \ \ \ \ \ \ \ \ \ \ \ \ \ \ \ \ \ \ \ \ \ \ \ \ \ \ }%
\left. +\left[ \beta \left( \frac{1}{2};q+1,2\right) \left\vert f^{\prime
\prime }\left( b\right) \right\vert ^{q}+\frac{1}{2^{q+2}\left( q+2\right) }%
\left\vert f^{\prime \prime }\left( a\right) \right\vert ^{q}\right] ^{\frac{%
1}{q}}\right\} .
\end{eqnarray*}
\end{proposition}

\begin{proposition}[Simpson inequality]
\label{pro3}Under the assumptions of Theorem \ref{ddd} with $\alpha =\frac{1%
}{2},$ $\lambda =\frac{1}{3}$, we get the following inequality:%
\begin{eqnarray*}
&&\left\vert \frac{1}{6}\left[ f\left( a\right) +4f\left( \frac{a+b}{2}%
\right) +f\left( b\right) \right] -\frac{1}{b-a}\int_{a}^{b}f\left( x\right)
dx\right\vert \\
&\leq &\left( b-a\right) ^{2}\left( \frac{2^{p+1}-1}{2^{2p+1}3^{p+1}\left(
p+1\right) }\right) ^{\frac{1}{p}}\left\{ \left[ \frac{1}{2^{q+2}\left(
q+2\right) }\left\vert f^{\prime \prime }\left( b\right) \right\vert
^{q}+\beta \left( \frac{1}{2};q+1,2\right) \left\vert f^{\prime \prime
}\left( a\right) \right\vert ^{q}\right] ^{\frac{1}{q}}\right. \\
&&\text{ \ \ \ \ \ \ \ \ \ \ \ \ \ \ \ \ \ \ \ \ \ \ \ \ \ \ \ \ \ \ \ \ \ \
\ \ }\left. +\left[ \beta \left( \frac{1}{2};q+1,2\right) \left\vert
f^{\prime \prime }\left( b\right) \right\vert ^{q}+\frac{1}{2^{q+2}\left(
q+2\right) }\left\vert f^{\prime \prime }\left( a\right) \right\vert ^{q}%
\right] ^{\frac{1}{q}}\right\} .
\end{eqnarray*}
\end{proposition}

\section{Applications to special means}

We now consider some applications with the following special means

a) The arithmetic mean:%
\begin{equation*}
A=A\left( a,b\right) :=\frac{a+b}{2},\text{\ \ }a,b\geq 0,
\end{equation*}

b) The geometic mean:%
\begin{equation*}
G=G\left( a,b\right) :=\sqrt{ab},\text{\ \ }a,b\geq 0,
\end{equation*}

c) The harmonic mean: 
\begin{equation*}
H=H\left( a,b\right) :=\frac{2ab}{a+b},\text{ \ }a,b>0,
\end{equation*}

d) The logarithmic mean: 
\begin{equation*}
L=L\left( a,b\right) :=\left\{ 
\begin{array}{c}
a\text{ \ \ \ \ \ \ \ \ \ if }a=b \\ 
\frac{b-a}{\ln b-\ln a}\text{ \ if }a\neq b%
\end{array}%
\right. ,\text{ \ }a,b>0,
\end{equation*}

e) The identric mean: 
\begin{equation*}
I=I\left( a,b\right) :=\left\{ 
\begin{array}{c}
a\text{ \ \ \ \ \ \ \ \ \ \ \ \ \ \ if }a=b \\ 
\frac{1}{e}\left( \frac{b^{b}}{a^{a}}\right) ^{\frac{1}{b-a}}\text{ \ if }%
a\neq b%
\end{array}%
\right. ,\text{ \ }a,b>0,
\end{equation*}

f) The $p-$logarithmic mean: 
\begin{equation*}
L_{p}=L_{p}\left( a,b\right) :=\left\{ 
\begin{array}{c}
a\text{ \ \ \ \ \ \ \ \ \ \ \ \ \ \ \ \ \ \ \ \ \ \ if }a=b \\ 
\left[ \frac{b^{p+1}-a^{p+1}}{\left( p+1\right) \left( b-a\right) }\right] ^{%
\frac{1}{p}}\text{ \ \ \ \ \ \ if }a\neq b%
\end{array}%
\right. ,\text{ }p\in 
\mathbb{R}
\backslash \left\{ -1,0\right\} ;\text{ \ }a,b>0.
\end{equation*}

We now derive some sophisticated bounds of the above means by using the
results at Section 3.

\begin{proposition}
Let $a,b\in 
\mathbb{R}
,0<a<b$. Then, for all $q\geq 1$ and $\frac{1}{p}+\frac{1}{q}=1,$ we have
\end{proposition}

\begin{eqnarray*}
&&\left\vert A^{-1}\left( a,b\right) -L^{-1}\left( a,b\right) \right\vert \\
&\leq &\frac{\left( b-a\right) ^{2}}{8}\left( \frac{1}{p+1}\right) ^{\frac{1%
}{p}}\left( \frac{1}{2\left( q+2\right) }\right) ^{\frac{1}{q}}\left\{
\dsum\limits_{i=1}^{2}\left[ \left( \frac{q+3}{q+1}\right) ^{i-1}\frac{1}{%
b^{3q}}+\left( \frac{q+3}{q+1}\right) ^{2-i}\frac{1}{a^{3q}}\right] ^{\frac{1%
}{q}}\right\} \\
&&\text{and} \\
&&\left\vert H^{-1}\left( a,b\right) -L^{-1}\left( a,b\right) \right\vert \\
&\leq &\frac{\left( b-a\right) ^{2}}{8}\left( \frac{2^{p+1}-1}{p+1}\right) ^{%
\frac{1}{p}}\left( \frac{1}{2\left( q+2\right) }\right) ^{\frac{1}{q}%
}\left\{ \dsum\limits_{i=1}^{2}\left[ \left( \frac{q+3}{q+1}\right) ^{i-1}%
\frac{1}{b^{3q}}+\left( \frac{q+3}{q+1}\right) ^{2-i}\frac{1}{a^{3q}}\right]
^{\frac{1}{q}}\right\} \\
&&\text{and} \\
&&\left\vert \frac{1}{3}H^{-1}\left( a,b\right) +\frac{2}{3}A^{-1}\left(
a,b\right) -L^{-1}\left( a,b\right) \right\vert \\
&\leq &\frac{\left( b-a\right) ^{2}}{24}\left( \frac{2^{p+1}-1}{3\left(
p+1\right) }\right) ^{\frac{1}{p}}\left( \frac{1}{2\left( q+2\right) }%
\right) ^{\frac{1}{q}}\left\{ \dsum\limits_{i=1}^{2}\left[ \left( \frac{q+3}{%
q+1}\right) ^{i-1}\frac{1}{b^{3q}}+\left( \frac{q+3}{q+1}\right) ^{2-i}\frac{%
1}{a^{3q}}\right] ^{\frac{1}{q}}\right\} .
\end{eqnarray*}

\begin{proof}
The assertions follow from Proposition \ref{pro1}, \ref{pro2} and \ref{pro3}
applied to convex mapping $f\left( x\right) =\frac{1}{x},$ $x\in \left[ a,b%
\right] $, respectively.
\end{proof}

\begin{proposition}
Let $a,b\in 
\mathbb{R}
,0<a<b$. Then, for all $q\geq 1$ and $\frac{1}{p}+\frac{1}{q}=1,$ we have
\end{proposition}

\begin{eqnarray*}
&&\left\vert \ln \left( A\left( a,b\right) \right) -\ln \left( I\left(
a,b\right) \right) \right\vert \\
&\leq &\frac{\left( b-a\right) ^{2}}{16}\left( \frac{1}{p+1}\right) ^{\frac{1%
}{p}}\left( \frac{1}{2\left( q+2\right) }\right) ^{\frac{1}{q}}\left\{
\dsum\limits_{i=1}^{2}\left[ \left( \frac{q+3}{q+1}\right) ^{i-1}\frac{1}{%
b^{2q}}+\left( \frac{q+3}{q+1}\right) ^{2-i}\frac{1}{a^{2q}}\right] ^{\frac{1%
}{q}}\right\} \\
&&\text{and} \\
&&\left\vert \ln \left( G\left( a,b\right) \right) -\ln \left( I\left(
a,b\right) \right) \right\vert \\
&\leq &\frac{\left( b-a\right) ^{2}}{16}\left( \frac{2^{p+1}-1}{p+1}\right)
^{\frac{1}{p}}\left( \frac{1}{2\left( q+2\right) }\right) ^{\frac{1}{q}%
}\left\{ \dsum\limits_{i=1}^{2}\left[ \left( \frac{q+3}{q+1}\right) ^{i-1}%
\frac{1}{b^{2q}}+\left( \frac{q+3}{q+1}\right) ^{2-i}\frac{1}{a^{2q}}\right]
^{\frac{1}{q}}\right\} \\
&&\text{and} \\
&&\left\vert \frac{1}{3}\ln \left( G\left( a,b\right) \right) +\frac{2}{3}%
\ln \left( A\left( a,b\right) \right) -\ln \left( I\left( a,b\right) \right)
\right\vert \\
&\leq &\frac{\left( b-a\right) ^{2}}{48}\left( \frac{2^{p+1}-1}{3\left(
p+1\right) }\right) ^{\frac{1}{p}}\left( \frac{1}{2\left( q+2\right) }%
\right) ^{\frac{1}{q}}\left\{ \dsum\limits_{i=1}^{2}\left[ \left( \frac{q+3}{%
q+1}\right) ^{i-1}\frac{1}{b^{2q}}+\left( \frac{q+3}{q+1}\right) ^{2-i}\frac{%
1}{a^{2q}}\right] ^{\frac{1}{q}}\right\} .
\end{eqnarray*}

\begin{proof}
The assertions follow from Proposition \ref{pro1}, \ref{pro2} and \ref{pro3}
applied to convex mapping $f\left( x\right) =\ln (x),$ $x\in \left[ a,b%
\right] $, respectively.
\end{proof}

\begin{proposition}
Let $a,b\in 
\mathbb{R}
,$ $0<a<b$. Then, for all $q\geq 1$, $\frac{1}{p}+\frac{1}{q}=1$ and $n\in 
\mathbb{N}
,$ $n>2$ we have
\end{proposition}

\begin{eqnarray*}
&&\left\vert A^{n}\left( a,b\right) -L_{n}^{n}\left( a,b\right) \right\vert
\\
&\leq &\frac{n\left( n-1\right) \left( b-a\right) ^{2}}{16}\left( \frac{1}{%
p+1}\right) ^{\frac{1}{p}}\left( \frac{1}{2\left( q+2\right) }\right) ^{%
\frac{1}{q}}\left\{ \dsum\limits_{i=1}^{2}\left[ \left( \frac{q+3}{q+1}%
\right) ^{i-1}b^{\left( n-2\right) q}+\left( \frac{q+3}{q+1}\right)
^{2-i}a^{\left( n-2\right) q}\right] ^{\frac{1}{q}}\right\} \\
&&\text{and} \\
&&\left\vert A\left( a^{n},b^{n}\right) -L_{n}^{n}\left( a,b\right)
\right\vert \\
&\leq &\frac{n\left( n-1\right) \left( b-a\right) ^{2}}{16}\left( \frac{%
2^{p+1}-1}{p+1}\right) ^{\frac{1}{p}}\left( \frac{1}{2\left( q+2\right) }%
\right) ^{\frac{1}{q}}\left\{ \dsum\limits_{i=1}^{2}\left[ \left( \frac{q+3}{%
q+1}\right) ^{i-1}b^{\left( n-2\right) q}+\left( \frac{q+3}{q+1}\right)
^{2-i}a^{\left( n-2\right) q}\right] ^{\frac{1}{q}}\right\} \\
&&\text{and} \\
&&\left\vert \frac{1}{3}A\left( a^{n},b^{n}\right) +\frac{2}{3}A^{n}\left(
a,b\right) -L_{n}^{n}\left( a,b\right) \right\vert \\
&\leq &\frac{n\left( n-1\right) \left( b-a\right) ^{2}}{48}\left( \frac{%
2^{p+1}-1}{3\left( p+1\right) }\right) ^{\frac{1}{p}}\left( \frac{1}{2\left(
q+2\right) }\right) ^{\frac{1}{q}}\left\{ \dsum\limits_{i=1}^{2}\left[
\left( \frac{q+3}{q+1}\right) ^{i-1}b^{\left( n-2\right) q}+\left( \frac{q+3%
}{q+1}\right) ^{2-i}a^{\left( n-2\right) q}\right] ^{\frac{1}{q}}\right\} .
\end{eqnarray*}

\begin{proof}
The assertions follow from Proposition \ref{pro1}, \ref{pro2} and \ref{pro3}
applied to convex mapping $f\left( x\right) =x^{n},$ $x\in \left[ a,b\right] 
$, respectively.
\end{proof}


\begin{thebibliography}{99}
\bibitem{SSDRPA} S.S. Dragomir and R.P. Agarwal, \textit{Two inequalities
for differentiable mappings and applications to special means of real
numbers and trapezoidal formula}, Appl. Math. Lett., 11(5) (1998), 91--95.

\bibitem{Drag1} S. S. Dragomir, \textit{Two mappings in connection to
Hadamard's inequalities}, J. Math. Anal. Appl. 167 (1992), 49--56.

\bibitem{Drag2} S. S. Dragomir, Y. J. Cho, and S. S. Kim, \textit{%
Inequalities of Hadamard's type for Lipschitzian mappings and their
applications}, J. Math. Anal. Appl. 245 (2000), 489--501.

\bibitem{Dragomir} S.S. Dragomir and C.E.M. Pearce, \textit{Selected Topics
on Hermite-Hadamard Inequalities and Applications}, RGMIA Monographs,
Victoria University, 2000. Online:[http://www.sta\textcurrency %
.vu.edu.au/RGMIA/monographs/hermite\_hadamard.html].

\bibitem{Hussain} S. Hussain, M.I. Bhatti and M. Iqbal, \textit{%
Hadamard-type inequalities for }s\textit{-convex functions I,} Punjab Univ.
Jour. of Math\textit{.}, Vol.41, pp:51-60, (2009).

\bibitem{USK} U.S. K\i rmac\i , \textit{Inequalities for differentiable
mappings and applications to special means of real numbers and to midpoint
formula}, Appl. Math. Comp\textit{.,} 147 (2004), 137-146.

\bibitem{USKMEO} U.S. K\i rmac\i\ and M.E. \"{O}zdemir, \textit{On some
inequalities for differentiable mappings and applications to special means
of real numbers and to midpoint formula}, Appl. Math. Comp.\textit{,} 153
(2004), 361-368.

\bibitem{K} U.S. K\i rmac\i , \textit{Improvement and further generalization
of inequalities for differentiable mappings and applications}, Computers and
Math. with Appl\textit{.,} 55 (2008), 485-493.

\bibitem{OAS} M. E. \"{O}zdemir, M. Avc\i\ and E. Set, \textit{On some
inequalities of Hermite--Hadamard type via m-convexity, }Appl. Math. Lett.
in press

\bibitem{CEMPJP} C.E.M. Pearce and J. Pe\v{c}ari\'{c}, \textit{Inequalities
for differentiable mappings with application to special means and quadrature
formulae}, Appl. Math. Lett., 13(2) (2000), 51--55.

\bibitem{yang} G. S. Yang and K. L. Tseng, \textit{On certain integral
inequalities related to Hermite--Hadamard inequalities}, J. Math. Anal.
Appl. 239 (1999), 180--187.

\bibitem{yeni} E. Set, I. Iscan, M. Z. Sarikaya and M. E. Ozdemir, \textit{%
On New Inequalities of Hermite-Hadamard-Fejer Type fror Convex Functions via
Fractional Integrals}, arXiv:1409.5243v1, submitted.

\bibitem{ana} M. Z. Sarikaya and N. Aktan, \textit{On the generalization of
some integral inequalities and their applications,} Math. and Comp.
Modelling., 54(2011), 2175-2182.

\bibitem{sarikaya} M. Z. Sarikaya, A. Saglam and H. Y\i ld\i r\i m, \textit{%
New inequalities of Hermite-Hadamard type for functions whose second
derivatives absolute values are convex and quasi-convex}, arXiv:1005.0451,
submited.

\bibitem{sarikaya1} M. Z. Sarikaya, E. Set and M. E. Ozdemir, \textit{On New
Inequalities of Simpson's Type for Functions whose Second Derivatives
Absolute Values are Convex, }RGMIA Res. Rep. Coll.,13(1) (2010), Supplement,
Article 1.

\bibitem{tunc} M. Tunc and S. Balgecti, Integral Inequalities for Mappings
Whose Derivatives are $s-$Convex in the Second Sense and Applications to
Special Means for Positive Real Numbers, arXiv:1407.1229v1, submitted.
\end{thebibliography}
\end{document}